\documentclass[12pt]{amsart}
\usepackage{amsmath}
\usepackage{amssymb}
\usepackage{tabularx}
\usepackage{enumerate}
\usepackage{graphicx}
\usepackage{texdraw}

\usepackage{mathrsfs}

\usepackage{amsfonts,amssymb,amsmath}
\usepackage{epsfig}

\newtheorem{Theorem}{Theorem}[section]
\newtheorem{Lemma}[Theorem]{Lemma}
\newtheorem{Proposition}[Theorem]{Proposition}

\numberwithin{equation}{section}
\numberwithin{figure}{section}

\begin{document}
	
\title[]{Existence of solutions to relativistic non-Abelian Chern-Simons-Higgs vortex equations on graphs}
\author[Y. Hu]{Yuanyang Hu$^1$}
\thanks{$^1$ School of Mathematics and Statistics, Henan University, Kaifeng, Henan 475004, P. R. China.}

\thanks{{\bf Emails:} {\sf yuanyhu@mail.ustc.edu.cn} (Y. Hu)}
\date{2021/8/14}

\begin{abstract}
Let $G=(V,E)$ be a connected finite graph.
We study the relativistic non-Abelian Chern-Simons-Higgs vortex equations on the graph $G$. We establish an existence result to the relativistic non-Abelian Chern-Simons-Higgs vortex equations.
\end{abstract}

\maketitle

\section{Introduction}	
Vortices play significant parts in many fields of theoretical physics including superconductivity theory, cosmology, condensed-matter physics,   electroweak theory, and quantum Hall effect. A tide of research related to vortex equations has been accomplished; see, for example, \cite{CI, T, TY, WY, Y} and the references therein. Recently, Han, Lin and Yang \cite{HLY2} investigated a system of relativistic non-Abelian Chern-Simons-Higgs vortex equations whose Cartan matrix $K$ is that of arbitrary simple Lie algebra, they established a general existence result for the doubly periodic solutions of the Chern-Simons-Higgs vortex equations.
 
In recent years, equations on graphs have attracted extensive attention; see, for example, \cite{ ALY, Hu, Hub, HWY, HLY, LP} and the references therein. 

 Recently, Huang, Lin and Yau \cite{HLY} proved the existence of solutions to mean field equations 
$$
\Delta u+e^{u}=\rho \delta_{0}
$$
and
$$
\Delta u=\lambda e^{u}\left(e^{u}-1\right)+4 \pi \sum_{j=1}^{M} \delta_{p_{j}}
$$
on graphs.

Let $G=(V,E)$ is a connected finite graph, and $V$ denotes the vetex set and $E$ denotes the edge set.

Inspired by the work of Huang-Lin-Yau \cite{HLY}, we study the relativistic Chern-Simons-Higgs equations  
\begin{equation}\label{1}
	\Delta u_{i}=\lambda\left(\sum_{j=1}^{n} \sum_{k=1}^{n} K_{k j} K_{j i} \mathrm{e}^{u_{j}} \mathrm{e}^{u_{k}}-\sum_{j=1}^{n} K_{j i} \mathrm{e}^{u_{j}}\right)+4 \pi \sum_{j=1}^{N_{i}} \delta_{p_{i j}}(x), \quad i=1, \ldots, n,
\end{equation}
on $G$, where $K=(K_{ij})$ is the Cartan matrix of a finite dimensional semisimple Lie algebra $L$, $n\ge1$ is the rank of $L$ which is the dimension of the Cartan subalgebra of $L$, $p_{ij}$, $j=1,...,N_{i}$, $i=1,...,n$, are arbitrarily chosen distinct vertices on the graph, and $\delta_{p_{ij}}$ is the Dirac mass at $p_{ij}$. With a view to handling the system in a unified framework, we need some suitable assumption on the matrix K. We suppose that
\begin{equation}\label{41}
	K^{T}=PS,
\end{equation}
$P$ is a diagonal matrix satisfying 
\begin{equation}\label{33}
	P:=diag\{P_1,...,P_n\},~P_{i}>0,~i=1,...,n,
\end{equation}
$S$ is a positive definite matrix of the form
\begin{equation}\label{43}
	S \equiv\left(\begin{array}{cccccc}
		\alpha_{11} & -\alpha_{12} & \cdots & \cdots & \cdots & -\alpha_{1 n} \\
		\vdots & \vdots & \vdots & \vdots & \vdots & \vdots \\
		-\alpha_{i 1} & -\alpha_{i 2} & \cdots & \alpha_{i i} & \cdots & -\alpha_{i n} \\
		\vdots & \vdots & \vdots & \vdots & \vdots & \vdots \\
		-\alpha_{n 1} & -\alpha_{n 2} & \cdots & \cdots & -\alpha_{n n-1} & \alpha_{n n}
	\end{array}\right),
\end{equation}
\begin{equation}\label{44}
	\alpha_{i i}>0, \quad i=1, \ldots, n, \quad \alpha_{i j}=\alpha_{j i} \geq 0, i \neq j=1, \ldots, n,
\end{equation}
and
\begin{equation}\label{45}
	\text{~all~the~entries~of~}S^{-1} \text{are positive}.
\end{equation}
By \eqref{45}, we conclude that 
\begin{equation}\label{46}
	R_{i} := \sum_{j=1}^{n}\left(\left(K^{\tau}\right)^{-1}\right)_{i j}>0
	\end{equation}
for $i=1,...,n$. 

 We now state our main results as follows.
 
 \begin{Theorem}\label{31}
 	Assume that the matrix $K$ satisfies \eqref{41}-\eqref{45}, then we have the following conclusions: 
 		
 	$(\mathrm{i})$~Suppose that \eqref{1} has a solution. Then we have 
 	\begin{equation}\label{47}
 		\lambda>\lambda_{0} \equiv \frac{16 \pi}{|V|} \frac{\sum \limits_{i=1}^{n} \sum\limits_{j=1}^{n} P_{i}^{-1}\left(K^{-1}\right)_{j i} N_{j}}{\sum\limits_{i=1}^{n} \sum\limits_{j=1}^{n} P_{i}^{-1}\left(K^{-1}\right)_{j i}}.
 	\end{equation}
 
 	$(\mathrm{ii})$~There exists a constant $\lambda_{1}>\lambda_{0}$ so that if $\lambda>\lambda_{1}$,  then \eqref{1} admits a solution $(u^{\lambda}_{1},...,u^{\lambda}_{n})$. 
 \end{Theorem}

The rest of the paper is arranged as below. In Section 2, We present some results that we will use frequently in the following pages. Section 3 and Section 4 are devoted to the proof of Theorem \ref{31}.

\section{Preliminary results}

For each edge $xy \in E$, We suppose that its weight $w_{xy}>0$ and that $w_{xy}=w_{yx}$. Set $\mu: V \to (0,+\infty)$ be a finite measure. For any function $u: V \to \mathbb{R}$, the Laplacian of $u$ is defined by 
\begin{equation}\label{l1}
	\Delta u(x)=\frac{1}{\mu(x)} \sum_{y \sim x} w_{y x}(u(y)-u(x)),
\end{equation}
where $y \sim x$ means $xy \in E$. The gradient  $\nabla$ of function $f$ is defined by a vector 
$$\nabla f (x):=\left(  \left[ f(y)-f(x)\right]  \sqrt{\frac{w_{xy}}{2\mu (x)}} \right)_{y\sim x} .$$ The gradient form of $u$ reads 
\begin{equation}
	\Gamma(u, v)(x)=\frac{1}{2 \mu(x)} \sum_{y \sim x} w_{x y}(u(y)-u(x))(v(y)-v(x)).
\end{equation}
We denote the length of the gradient of $u$ by
\begin{equation*}
	|\nabla u|(x)=\sqrt{\Gamma(u)(x)}=\left(\frac{1}{2 \mu(x)} \sum_{y \sim x} w_{x y}(u(y)-u(x))^{2}\right)^{1 / 2}.
\end{equation*}
Denote, for any function $
u: V \rightarrow \mathbb{R}
$, an integral of $u$ on $V$ by $\int \limits_{V} u d \mu=\sum\limits_{x \in V} \mu(x) u(x)$. Denote  $|V|$=$ \text{Vol}(V)=\sum \limits_{x \in V} \mu(x)$ the volume of $V$. For $p > 0$, denote $|| u ||_{p}:=||u||_{L^{p}(V)}=(\int \limits_{V} |u|^{p} d \mu)^{\frac{1}{p}}$. Define a sobolev space and a norm on it by 
\begin{equation*}
	W^{1,2}(V)=\left\{u: V \rightarrow \mathbb{R}: \int \limits_{V} \left(|\nabla u|^{2}+u^{2}\right) d \mu<+\infty\right\},
\end{equation*}
and \begin{equation*}
	\|u\|_{H^{1}(V)}=	\|u\|_{W^{1,2}(V)}=\left(\int \limits_{V}\left(|\nabla u|^{2}+u^{2}\right) d \mu\right)^{1 / 2}.
\end{equation*}

To apply the variational method, we need the following Sobolev embedding, Truding-Moser inequlity and interpolation inequality on graphs.

\begin{Lemma}\label{21}
	{\rm (\cite[Lemma 5]{ALY})} Let $G=(V,E)$ be a finite graph. The sobolev space $W^{1,2}(V)$ is precompact. Namely, if ${u_j}$ is bounded in $W^{1,2}(V)$, then there exists some $u \in W^{1,2}(V)$ such that up to a subsequence, $u_j \to u$ in $W^{1,2}(V)$.
\end{Lemma}

\begin{Lemma}\label{2.2}
	{\rm (\cite[Lemma 6]{ALY})}	Let $G = (V, E)$ be a finite graph. For all functions $u : V \to \mathbb{R}$ with $\int \limits_{V} u d\mu = 0$, there 
	exists some constant $C$ depending only on $G$ such that $\int \limits_{V} u^2 d\mu \le C \int \limits_{V} |\nabla u|^2 d\mu$.
\end{Lemma} 

\begin{Lemma}\label{mt}
	{\rm (\cite[Lemma 7]{ALY})}
	Let $G=(V,E)$ be a finite graph. For any $\beta\in \mathbb{R}$, there exists a constant $C$ depending only on $\beta$ and $G$ such that for all functions $v$ with $\int\limits_{V} |\nabla v|^{2} d \mu\le 1$  and $\int\limits_{V} v d \mu =0$, there holds 
	\begin{equation}
		\int\limits_{V} e^{\beta v^{2}} d\mu \le C.
	\end{equation}
\end{Lemma} 

\begin{Lemma}\label{i}
	{\rm (Interpolation inequality for $L^{r}$-norms on graphs) }. Suppose that $\theta \in (0,1) $, $0<\theta r\le s$, $0<(1-\theta)r \le t$ and $\frac{1}{r}=\frac{\theta}{s}+\frac{1-\theta}{t}$. Then we have 
	\begin{equation}
		||u||_{L^{r}(V)}\le ||u||^{\theta}_{L^{s}(U)} ||u||^{1-\theta}_{L^{t}(U)} .
	\end{equation}
\end{Lemma}
\begin{proof}
By H$\ddot{\text{o}}$lder inequality, we see that  
	\begin{equation}
		\begin{aligned}
			\int\limits_{V}|u|^{r} d \mu &=\int\limits_{V}|u|^{\theta r}|u|^{(1-\theta) r} d \mu \\
			& \leq\left(\int\limits_{V}|u|^{\theta r \frac{s}{\theta r}} d \mu \right)^{\frac{\theta r}{s}}\left(\int\limits_{V}|u|^{(1-\theta) r \frac{t}{(1-\theta) r}} d \mu \right)^{\frac{(1-\theta) r}{t}} .
		\end{aligned}
	\end{equation}
\end{proof}

In order to establish Theorem \ref{31}, we need the following result due to Huang-Lin-Yau.
\begin{Theorem}\label{ly}
	{\rm (\cite[Theorem 2.2]{HLY})}
	There is a critical value $\lambda_{c}$ depending on $G$ satisfying	
	$$\lambda_{c}\ge \frac{16\pi M}{|V|},$$
	such that when $\lambda>\lambda_{c}$, the equation
	\begin{equation}\label{13a}
		\Delta u=\lambda e^{u}\left(e^{u}-1\right)+4 \pi \sum_{j=1}^{M} \delta_{p_{j}},~x\in G,
	\end{equation}
	has a solution $u_{\lambda}$ on $G$, and when $\lambda<\lambda_{c}$, the equation \eqref{13a} has no solution, where $M>0$ is an integer.
\end{Theorem}

In fact, we may establish the following more accurate result.  
\begin{Proposition}
	The solution $u_{\lambda}$ obtained in Theorem \ref{ly} is maximal in the sense that if $u$ is any other solution of \eqref{1}, then 
	\begin{equation}
		u\le u_{\lambda}.
	\end{equation} 
\end{Proposition}

\begin{proof}
	 Suppose that $u$ is any other solution of \eqref{13a}. Then it is clear that $u$ is a subsolution and hence 
	\begin{equation}\label{b}
		u\le v_{n}
	\end{equation} 
	as in the proof of Lemma 4.2 in \cite{HLY} for every $n\in \mathbb{N}$, where $v_{n}$ is defined by iterative scheme (4.5) in \cite{HLY}. Letting $n\to +\infty$ in \eqref{b}, we deduce that $u\le u_{\lambda}$ and hence that $u_{\lambda}$ is a maximal solution of \eqref{13a}. 
	
	We now complete the proof.
\end{proof}

Furthermore, we have the following propositions.

\begin{Proposition}
	If $\lambda_{1}>\lambda_{2}>\lambda_{c}$, then $u_{\lambda_{1}}\ge u_{\lambda_{2}}$.
\end{Proposition}
\begin{proof}
	By Lemma 4.4 of \cite{HLY}, we deduce that $u_{\lambda_{2}}<0$, and hence that 
	\begin{equation}
		\begin{aligned}
			\Delta u_{\lambda_{2}}&=\lambda_{2} e^{u_{\lambda_{2}}}(e^{u_{\lambda_{2}}}-1)+4\pi\sum_{j=1}^{M} \delta_{p_{j}} \\
			& >\lambda_{1} e^{u_{\lambda_{2}}}(e^{u_{\lambda_{2}}}-1) +4\pi\sum_{j=1}^{M} \delta_{p_{j}}.
		\end{aligned}
	\end{equation}
Thus $u_{\lambda_{2}}$ is a subsolution of \eqref{13a} with $\lambda=\lambda_{1}$. By the sub-supersolution argument as in the proof of Lemma 4.2 in \cite{HLY}, and the maximality of $u_{\lambda_{1}}$, this implies that $u_{\lambda_{2}} \le u_{\lambda_{1}}$.
\end{proof}

\begin{Proposition}\label{f2}
	 Let $u_{\lambda}$ be the maximal solution of \eqref{13a} for $\lambda>\lambda_{c}$. We have 
	\begin{equation}
		u_{\lambda} \to 0\text{~as~}\lambda\to +\infty\text{~uniformly~on~}V.
	\end{equation}
\end{Proposition}
\begin{proof}
Fix $\lambda_{0}>\lambda_{c}$, since $u_{\lambda}$ is monotone increasing in $\lambda$, by Lemma 4.4 of \cite{HLY}, we deduce that $u_{\lambda_{0}} \le u_{\lambda}\le 0$ and
	\begin{equation}
		\int\limits_{V} e^{u_{\lambda}} (1-e^{ u_{\lambda}}) d \mu= \frac{4\pi M}{\lambda},
	\end{equation} 
for $\lambda\ge \lambda_{0}$.
	Let $\bar{v}(x):=\sup\limits_{\lambda>\lambda_{0}} u_{\lambda} (x) $, $x\in V$. We deduce that $u_{\lambda_{0}} \le \bar{v} \le 0$ in $V$ and 
	\begin{equation}
		\int\limits_{V} e^{ \bar{v}} (1-e^{ \bar{v}}) d \mu= 0.
	\end{equation} 
It follows that $\bar{v}\equiv 0$ on $V$.

We now complete the proof.
\end{proof}

\section{The constraints}

For the sake of convience, by applying the translation
\begin{equation}\label{51}
	u_{i} \to u_{i}+ \text{ln} R_{i},~i=1,...,n,
\end{equation}
in equations \eqref{1}, we can conclude that
\begin{equation}
	\Delta u_{i}=\lambda\left(\sum_{j=1}^{n} \sum_{k=1}^{n} \tilde{K}_{j k} \tilde{K}_{i j} \mathrm{e}^{u_{j}} \mathrm{e}^{u_{k}}-\sum_{j=1}^{n} \tilde{K}_{i j} \mathrm{e}^{u_{j}}\right)+4 \pi \sum_{j=1}^{N_{i}} \delta_{p_{i j}}(x)
\end{equation}
for $i=1,...,n$.
Set $u_{i}^{0}$ be the unique solution to 
\begin{equation}
	\Delta u_{i}^{0}=4 \pi \sum_{s=1}^{N_{i}} \delta_{p_{i s}}-\frac{4 \pi N_{i}}{|V|}, \quad \int\limits_{V} u_{i}^{0} \mathrm{~d} x=0.
\end{equation}
Denote $u_{i}=u_{i}^{0}+v_{i}$, $i=1,...,n$; then $v_{i} (i=1,...,n)$  satisfy
\begin{equation}\label{59}
	\Delta v_{i}=\lambda\left(\sum_{j=1}^{n} \sum_{k=1}^{n} \tilde{K}_{j k} \tilde{K}_{i j} \mathrm{e}^{u_{j}^{0}+v_{j}} \mathrm{e}^{u_{k}^{0}+v_{k}}-\sum_{j=1}^{n} \tilde{K}_{i j} \mathrm{e}^{u_{j}^{0}+v_{j}}\right)+\frac{4 \pi N_{i}}{|V|},
\end{equation}
or
\begin{equation}\label{510}
	\Delta \mathbf{v}=\lambda \tilde{K} U \tilde{K}(\mathbf{U}-\mathbf{1})+\frac{4 \pi \mathbf{N}}{|V|},
\end{equation}
where  $\mathbf{v}=(v_1,...,v_n)^{T},$ $\mathbf{N}=(N_1,...,N_n)^{T}$, $U=diag\{e^{u_{1}^{0}+v_{1}},...,e^{u_{n}^{0}+v_{n}}\},$ $\mathbf{U}=(e^{u_{1}^{0}+v_{1}},...,e^{u_{n}^{0}+v_{n}})^{T},$
\begin{equation}
	\tilde{K}:=K^{T}R=PSR,~R:=diag\{R_1,...,R_n \}.
\end{equation}

We next establish a necessary condition for the existence of solutions of \eqref{1}, and then the conclusion (i) of Theorem \ref{31} follows.

\begin{Lemma}
	Suppose that \eqref{1} admits a solution. Then 
	\begin{equation}\label{4.7}
		\lambda>\lambda_{0} := \frac{16 \pi}{|V|} \frac{\sum \limits_{i=1}^{n} \sum\limits_{j=1}^{n} P_{i}^{-1}\left(K^{-1}\right)_{j i} N_{j}}{\sum\limits_{i=1}^{n} \sum\limits_{j=1}^{n} P_{i}^{-1}\left(K^{-1}\right)_{j i}}.
	\end{equation} 
\end{Lemma}
\begin{proof}
Denote \begin{equation}\label{57}
	A = P^{-1} S^{-1} P^{-1} \quad \text { and } \quad Q = R S R.
\end{equation}	
In view of $S$ is positive definite, we see that $A$ ang $Q$ are positive definite. By \eqref{45}, we deduce that \begin{equation}
	\mathbf{b} \equiv\left(b_{1}, \ldots, b_{n}\right)^{T} := 4 \pi A \mathbf{N}=4 \pi P^{-1} S^{-1} P^{-1} \mathbf{N}>0.
\end{equation} 
From \eqref{510}, we deduce that 
\begin{equation}
	\Delta A \mathbf{v}=\lambda U Q(\mathbf{U}-\mathbf{1})+\frac{\mathbf{b}}{|V|}.
\end{equation} 
It follows that \begin{equation}\label{515}
	\int\limits_{V} \mathrm{UQ}(\mathbf{U}-\mathbf{1}) \mathrm{d} \mu +\frac{\mathbf{b}}{\lambda}=\mathbf{0},
\end{equation}
where $\mathbf{1}:=(1,...,1)^{T}$.
Multiplying both sides of \eqref{515} by $\mathbf{1}^{T}$, we deduce that 
\begin{equation}\label{516}
	\int_{\Omega} \mathbf{U}^{T} Q(\mathbf{U}-\mathbf{1}) \mathrm{d} \mu+\frac{\mathbf{1}^{T} \mathbf{b}}{\lambda}=0.
\end{equation}
From \eqref{41} and \eqref{46}, we conclude that 
\begin{equation}\label{518}
	\left(K^{T}\right)^{-1} \mathbf{1}=S^{-1} P^{-1} \mathbf{1}=R \mathbf{1}.
\end{equation}
Combining \eqref{516} and \eqref{518}, we deduce that
\begin{equation}\label{517}	
		\int\limits_{V}\left(\mathbf{U}-\frac{\mathbf{1}}{2}\right)^{\tau} Q\left(\mathbf{U}-\frac{\mathbf{1}}{2}\right) \mathrm{d} \mu 
		=\frac{|V|}{4} \mathbf{1}^{\tau} P^{-1}\left(K^{\tau}\right)^{-1} \mathbf{1}-\frac{4 \pi \mathbf{1}^{\tau} P^{-1}\left(K^{\tau}\right)^{-1} \mathbf{N}}{\lambda}.
\end{equation}
Recall that $Q$ is positive define, it follows from \eqref{517} that
\begin{equation}
	\frac{|V|}{4} \mathbf{1}^{\tau} P^{-1}\left(K^{\tau}\right)^{-1} \mathbf{1}-\frac{4 \pi \mathbf{1}^{\tau} P^{-1}\left(K^{\tau}\right)^{-1} \mathbf{N}}{\lambda}>0,
\end{equation}
which implies that \eqref{4.7} holds. 
\end{proof}

\section{The proof of Theorem \ref{31}}
In this section, we formulate a variational  solution of equations \eqref{1} by using an equality type constraint. 
Define the energy functional by 
\begin{equation}\label{19}
	I(\mathbf{v})=\frac{1}{2} \sum_{j,k=1}^{n} \int\limits_{V}   b_{kj} \Gamma(v_{k},v_{j})  d \mu+\frac{\lambda}{2} \int\limits_{V}(\mathbf{U}-\mathbf{1})^{T} Q(\mathbf{U}-\mathbf{1}) \mathrm{d} \mu+\int\limits_{V} \frac{\mathbf{b}^{T} \mathbf{v}}{|V|} \mathrm{d} \mu,
\end{equation}
where $A=(b_{ij})_{n\times n}$. Due to the fact that $Q$ and $A$ are symmetric, we know that if $\mathbf{v}$ is a critical point to $I$, then it is a solution to \eqref{1}.

 We could work on the standard space $H^{1}(V):=W^{1,2}(V)$. Denote 
\begin{equation}
	H^{0}:=\{v\in W^{1,2}(V) | \int\limits_{V} v d \mu=0 \}.
\end{equation}
Clearly, for any $f\in H$, there exists a unique $c\in \mathbb{R}$ and $f^{'}\in H^{0}$ such that 
\begin{equation}\label{29}
	f=c+f^{'}
\end{equation}
In the sequal, we use $H^{1}(V)$ to denote the spaces of both scalar and vector-valued functions. 

Suppose that $\mathbf{v}=\mathbf{w}+\mathbf{c} \in H^{1}(V)$ given in \eqref{29} satisfies \eqref{515}, we deduce that \begin{equation}\label{61}
	\operatorname{diag}\left\{\mathrm{e}^{c_{1}}, \ldots, \mathrm{e}^{c_{n}}\right\} \tilde{Q}\left(\begin{array}{c}
		\mathrm{e}^{c_{1}} \\
		\vdots \\
		\mathrm{e}^{c_{n}}
	\end{array}\right)-P^{-1} R \operatorname{diag}\left\{a_{1}, \ldots, a_{n}\right\}\left(\begin{array}{c}
		\mathrm{e}^{c_{1}} \\
		\vdots \\
		\mathrm{e}^{c_{n}}
	\end{array}\right)+\frac{\mathbf{b}}{\lambda}=\mathbf{0},
\end{equation} 
where \begin{equation}\label{64}
		a_{i}  := a_{i}\left(w_{i}\right) = \int\limits_{V} \mathrm{e}^{u_{i}^{0}+w_{i}} \mathrm{~d} \mu ,
	\end{equation}
		 \begin{equation}\label{65}
		a_{i j}  := a_{i j}\left(w_{i}, w_{j}\right) = \int\limits_{V} \mathrm{e}^{u_{i}^{0}+u_{j}^{0}+w_{i}+w_{j}} \mathrm{~d} \mu, \quad i, j=1, \ldots, n,
\end{equation}
\begin{equation}\label{62}
	\tilde{Q} := \tilde{Q}(\mathbf{w}) = R \tilde{S} R,
\end{equation}
and
\begin{equation}\label{63}
	\tilde{S} \equiv\left(\begin{array}{cccccc}
		a_{11} \alpha_{11}  & -\alpha_{12} a_{12} & \cdots & \cdots & \cdots & -\alpha_{1 n} a_{1 n} \\
		\vdots & \vdots & \vdots & \vdots & \vdots & \vdots \\
		-\alpha_{i 1} a_{i 1} & -\alpha_{i 2} a_{i 2} & \cdots & \alpha_{i i} a_{i i} & \cdots & -\alpha_{i n} a_{i n} \\
		\vdots & \vdots & \vdots & \vdots & \vdots & \vdots \\
		-a_{n 1} \alpha_{n 1}  & -\alpha_{n 2} a_{n 2} & \cdots & \cdots & \cdots & a_{n n} \alpha_{n n} 
	\end{array}\right).
\end{equation}
By \eqref{62} and \eqref{63}, we deduce that 
\begin{equation}
	\tilde{Q}~\text{ is~positive~ positive~definite}. 
\end{equation}
We now write \eqref{61} as the component form:
\begin{equation}\label{68}
	\mathrm{e}^{2 c_{i}} R_{i}^{2} \alpha_{i i} a_{i i}-\mathrm{e}^{c_{i}}\left(\frac{R_{i} a_{i}}{P_{i}}+\sum_{j \neq i} \mathrm{e}^{c_{j}} R_{i} R_{j} \alpha_{i j} a_{i j}\right)+\frac{b_{i}}{\lambda}=0, \quad i=1, \ldots, n.
\end{equation}
Of course, \eqref{59} is a quadratic equation in $t=e^{c}$ which admits a solution if and only if 
\begin{equation}\label{67}
	\left(\frac{R_{i} a_{i}}{P_{i}}+\sum_{j \neq i} \mathrm{e}^{c_{j}} R_{i} R_{j} \alpha_{i j} a_{i j}\right)^{2} \geq \frac{4 R_{i}^{2} b_{i} \alpha_{i i} a_{i i}}{\lambda}, \quad i=1, \ldots, n.
\end{equation}
It is clear that \eqref{68} follows from the following inequalities
\begin{equation}\label{69}
	\frac{a_{i}^{2}}{a_{i i}} \geq \frac{4 \alpha_{i i} P_{i}^{2} b_{i}}{\lambda}, \quad i=1, \ldots, n.
\end{equation}
Denote 
\begin{equation}\label{A}
	\mathscr{A} \equiv\left\{\mathbf{w} \mid \mathbf{w} \in {H}^{0}(V) \text { such that } \eqref{69} \text { holds }\right\}.
\end{equation}
In this case, we may select $\mathbf{c}=\mathbf{c}(\mathbf{w}):=(c_1,...,c_n)$ in \eqref{68} to satisfy 
\begin{equation}\label{611}
	\begin{aligned}
		\mathrm{e}^{c_{i}}=& \frac{1}{2 R_{i}^{2} \alpha_{i i} a_{i i}}\left\{\left(\frac{R_{i} a_{i}}{P_{i}}+\sum_{j \neq i} \mathrm{e}^{c_{j}} R_{i} R_{j} \alpha_{i j} a_{i j}\right)\right.\\
		&+\sqrt{\left.\left(\frac{R_{i} a_{i}}{P_{i}}+\sum_{j \neq i} \mathrm{e}^{c_{j}} R_{i} R_{j} \alpha_{i j} a_{i j}\right)^{2}-\frac{4 b_{i} R_{i}^{2} \alpha_{i i} a_{i i}}{\lambda}\right\}} \\
		& =: f_{i}\left(\mathrm{e}^{c_{1}}, \ldots, \mathrm{e}^{c_{n}}\right), \quad i=1, \ldots, n .
	\end{aligned}
\end{equation}

To proof Lemma \ref{u8} and Lemma \ref{f1}, we give a priori estimates.

\begin{Lemma}
	For any $\mathbf{w} \in \mathscr{A}$ and $\epsilon \in [0,1]$, if $\mathbf{t}$ satisfies the following equations
	\begin{equation}\label{72}
		\mathbf{F}(\epsilon, \mathbf{t}) \equiv \mathbf{t}-\mathbf{f}(\epsilon, \mathbf{t})=\mathbf{0}, \quad \mathbf{t} \in \mathbb{R}_{+}^{n}, \quad \epsilon \in[0,1],
	\end{equation}
where 
\begin{equation}\label{73}
	\mathbf{f}(\epsilon, \mathbf{t}) \equiv\left(f_{1}(\epsilon, \mathbf{t}), \ldots, f_{n}(\epsilon, \mathbf{t})\right)^{T},
\end{equation}
and
\begin{equation}\label{74}
	\begin{aligned}
		f_{i}(\epsilon, \mathbf{t}) \equiv & \frac{1}{2 R_{i}^{2} \alpha_{i i} a_{i i}}\left\{\left(\frac{R_{i} a_{i}}{P_{i}}+\sum_{j \neq i} t_{j} R_{i} R_{j} \alpha_{i j} a_{i j}\right)\right.\\
		+&\left.\sqrt{\left(\frac{R_{i} a_{i}}{P_{i}}+\sum_{j \neq i} t_{j} R_{i} R_{j} \alpha_{i j} a_{i j}\right)-\frac{4 \epsilon b_{i} R_{i}^{2} \alpha_{i i} a_{i i}}{\lambda}}\right\} \\
		i &=1, \ldots, n,
	\end{aligned}
\end{equation}
then 
\begin{equation}\label{76}
	0<a_{i} t_{i} \leq|V|, \quad i=1, \ldots, n
\end{equation}
and
\begin{equation}\label{77}
	0< t_{i} \leq 1, \quad i=1, \ldots, n.
\end{equation}
\end{Lemma}

\begin{proof}
	From \eqref{72}, \eqref{73} and \eqref{74}, we deduce that the left hand side of \eqref{76} and \eqref{77} holds.
	
	By \eqref{72}, we see that 
	\begin{equation}
		t_{i}=f_{i}(\epsilon, \mathbf{t}) \leq \frac{\frac{R_{i} a_{i}}{P_{i}}+\sum_{j \neq i} t_{j} R_{i} R_{j} \alpha_{i j} a_{i j}}{R_{i}^{2} \alpha_{i i} a_{i i}}, \quad i=1, \ldots, n,
	\end{equation}
which implies that 
\begin{equation}\label{k1}
	\tilde{Q} \mathbf{t} \leq P^{-1} R \mathbf{a},
\end{equation}
where \begin{equation}
	\mathbf{a} :=\left(a_{1}, \ldots, a_{n}\right)^{T}.
\end{equation}
From \eqref{64}, \eqref{65} and
 H$\ddot{\text{o}}$lder inequality, we conclude that 
 \begin{equation}\label{713}
 	a_{i j}^{2} \leq a_{i i} a_{j j}, \quad a_{i} \leq|V|^{\frac{1}{2}} a_{i i}^{\frac{1}{2}}, \quad i, j=1, \ldots, n.
 \end{equation}
 By $\eqref{713}$ and the fact that $t_{i}>0,$ for $i=1,...,n$, we conclude that 
 \begin{equation}\label{714}
 	\operatorname{diag}\left\{a_{11}^{\frac{1}{2}}, \ldots, a_{n n}^{\frac{1}{2}}\right\} Q \operatorname{diag}\left\{a_{11}^{\frac{1}{2}}, \ldots, a_{n n}^{\frac{1}{2}}\right\} \mathbf{t} \leq \tilde{Q} \mathbf{t}.
 \end{equation}
From \eqref{45} and \eqref{57}, we conclude that \begin{equation}\label{715}
	(Q^{-1})_{ij}>0,~i,j=1,...,n.
\end{equation}
 Thus, combining \eqref{713}, \eqref{714} and \eqref{715}, we conclude that 
 \begin{equation}
 	\begin{aligned}\label{716}
 		\mathbf{t} & \leq \operatorname{diag}\left\{a_{11}^{-\frac{1}{2}}, \ldots, a_{n n}^{-\frac{1}{2}}\right\} Q^{-1} \operatorname{diag}\left\{a_{11}^{-\frac{1}{2}}, \ldots, a_{n n}^{-\frac{1}{2}}\right\} P^{-1} R \mathbf{a} \\
 		&=\operatorname{diag}\left\{a_{11}^{-\frac{1}{2}}, \ldots, a_{n n}^{-\frac{1}{2}}\right\} Q^{-1} \operatorname{diag}\left\{a_{1} a_{11}^{-\frac{1}{2}}, \ldots, a_{n} a_{n n}^{-\frac{1}{2}}\right\} P^{-1} R \mathbf{1} .
 	\end{aligned}
 \end{equation}
 Therefore, by \eqref{518}, \eqref{713}, \eqref{715} and \eqref{716}, we deduce that 
 \begin{equation}\label{717}
 	\begin{aligned}
 		&\operatorname{diag}\left\{a_{1}, \ldots, a_{n}\right\} \mathbf{t} \\
 		&\leq \operatorname{diag}\left\{a_{1} a_{11}^{-\frac{1}{2}}, \ldots, a_{n} a_{n n}^{-\frac{1}{2}}\right\} Q^{-1} \operatorname{diag}\left\{a_{1} a_{11}^{-\frac{1}{2}}, \ldots, a_{n} a_{n n}^{-\frac{1}{2}}\right\} P^{-1} R \mathbf{1} \\
 		&\leq|V| Q^{-1} P^{-1} R \mathbf{1} \\
 		&=|V| \mathbf{1}.
 	\end{aligned}
 \end{equation}
 Thus, we get the right hand side of \eqref{76} holds. It follows from Jensen's inequality that $$\frac{a_i}{|V|}\ge e^ { \frac{\int\limits_{V} u_{1}^{0}+w_{i} d \mu }{|V|} },~i=1,...,n.$$ Then the right hand sides of \eqref{77} follows from this and \eqref{717}.
	
	We now complete the proof.
\end{proof}

The following result implies that we can solve constraints \eqref{611}, so constraints \eqref{68} could be solved.
\begin{Lemma}\label{434}
	For any $\mathbf{w} \in$ $ \mathscr{A}$, the equations 
	\begin{equation}\label{82}
		\mathbf{F}(\mathbf{t}) \equiv \mathbf{t}-\mathbf{f}(\mathbf{t})=\mathbf{0}, \quad \mathbf{t} \equiv\left(t_{1}, \ldots, t_{n}\right)^{T} \in \mathbb{R}_{+}^{n},
	\end{equation}
	 admits a solution $\mathbf{t} \in (0,\infty)^{n}$, where $\mathbb{R}_{+}^{n} \equiv\left(\mathbb{R}_{+}\right)^{n}, \mathbf{f}(\mathbf{t}) \equiv\left(f_{1}(\mathbf{t}), \ldots, f_{n}(\mathbf{t})\right)^{T}$.	
\end{Lemma}
\begin{proof}
	For the sake of convience, we write 
	\begin{equation}
		\left(\alpha_{1}, \ldots, \alpha_{n}\right)^{\tau}<(\leq)\left(\beta_{1}, \ldots, \beta_{n}\right)^{\tau} \text { if } \alpha_{i}<(\leq) \beta_{i}, i=1, \ldots, n,
	\end{equation}
and we use the same notation for matrices. We next find a solution to \eqref{82} with $\epsilon=1$.

 By \eqref{77}, we conclude that $\mathbf{F}(\epsilon, \mathbf{t})$ has no zero on the boundary of $\Omega$ for all $\mathbf{w} \in \mathcal{A}$ and $\epsilon \in[0,1]$, where $\Omega:=(0,r_0)^{n}$ and $r_0>1$ is a constant. Thus, we can define the Brouwer degree $\operatorname{deg}(\mathbf{F}(\epsilon, \mathbf{t}),\Omega, \mathbf{0}).$ Clearly, 
\begin{equation}\label{429}
	\mathbf{F}(0, \mathbf{t})=\mathbf{0}
\end{equation}
is equivalent to 
\begin{equation}\label{237}
	t_{i}-\frac{\frac{R_{i} a_{i}}{P_{i}}+\sum\limits_{j \neq i} t_{j} R_{i} R_{j} \alpha_{i j} a_{i j}}{R_{i}^{2} \alpha_{i i} a_{i i}}=0, \quad i=1, \ldots, n.
\end{equation}
We write \eqref{237} in its vector form
\begin{equation}
	\tilde{Q} \mathbf{t}=P^{-1} R \mathbf{a}.
\end{equation}
Since $\tilde{Q}$ is invertible, we know that \eqref{429} admits a unique solution \begin{equation}
	\mathbf{t}=\tilde{Q}^{-1} P^{-1} R \mathbf{a},
\end{equation}
By \eqref{77},we see that it belong to the interior of $\Omega$. By the fact that $\tilde{Q}$ is positive definite, we deduce that the Jacobian of $\mathbf{F}(0, \mathbf{t})$ is positive everywhere, and hence that $\operatorname{deg}(\mathbf{F}(0, \mathbf{t}),\Omega, \mathbf{0})=1$. It is easy to check that $\mathbf{F}(\epsilon, \mathbf{t})$ is a smooth function for any $\epsilon\in [0,1]$. Thus by homotopy invariance, we deduce that 
\begin{equation}
	\operatorname{deg}(\mathbf{F}(1, \mathbf{t}), \Omega, \mathbf{0})=\operatorname{deg}(\mathbf{F}(0, \mathbf{t}), \Omega, \mathbf{0}).
\end{equation}
 
 Now we complete the proof.
\end{proof}
 The following Lemma follows from Lemma \ref{434} immediately. 
 \begin{Lemma}\label{4}
 	For any $\mathbf{w} \in \mathscr{A}$, \eqref{68} admits a solution   $\mathbf{c}(\mathbf{w})=(c_{1}(\mathbf{w}),...,c_{n}(\mathbf{w}))^{T}$ which satisfies \eqref{611}, so that $\mathbf{v}=\mathbf{w}+\mathbf{c}(\mathbf{w})=(w_1+c_{1}(\mathbf{w}),...,w_n+c_{n}(\mathbf{w}))^{T}$ satisfies \eqref{515}.
 \end{Lemma} 

Define the constrained functional 
\begin{equation}\label{J}
	J(\mathbf{w}) := I(\mathbf{w}+\mathbf{c}(\mathbf{w})), \quad \mathbf{w} \in \mathscr{A}.
\end{equation}
  For all $\mathbf{w} \in \mathscr{A}$, since $\mathbf{v}=\mathbf{w}+\mathbf{c}(\mathbf{w})$ satisfies \eqref{515}, we conclude that 
  \begin{equation}
  	\begin{aligned}
  		\int\limits_{V}(\mathbf{U}-\mathbf{1})^{\tau} Q(\mathbf{U}-\mathbf{1}) \mathrm{d} \mu &=\int\limits_{V} \mathbf{1}^{T} Q(\mathbf{1}-\mathbf{U}) \mathrm{d} \mu-\frac{\mathbf{1}^{T} \mathbf{b}}{\lambda} \\
  		&=\int\limits_{V} \mathbf{1}^{T} P^{-1} R(\mathbf{1}-\mathbf{U}) \mathrm{d} \mu-\frac{\mathbf{1}^{\tau} \mathbf{b}}{\lambda}.
  	\end{aligned}
  \end{equation}
By \eqref{19}, we deduce that 
\begin{equation}
	\begin{aligned}
		J(\mathbf{w})=& \frac{1}{2} \sum_{j,k=1}^{n} \int\limits_{V}   b_{kj} \Gamma(v_{k},v_{j})  d \mu+\frac{\lambda}{2} \mathbf{1}^{T} P^{-1} R \int\limits_{V}(\mathbf{1}-\mathbf{U}) \mathrm{d} \mu+\mathbf{b}^{T} \mathbf{c}-\frac{\mathbf{1}^{T} \mathbf{b}}{2} \\
		=& \frac{1}{2} \sum_{j,k=1}^{n} \int\limits_{V}   b_{kj} \Gamma(v_{k},v_{j})  d \mu+\frac{\lambda}{2} \sum_{i=1}^{n} \frac{R_{i}}{P_{i}} \int\limits_{V}\left(1-\mathrm{e}^{c_{i}} \mathrm{e}^{u_{i}^{0}+w_{i}}\right) \mathrm{d} \mu \\
		&+\sum_{i=1}^{n} b_{i} c_{i}-\frac{1}{2} \sum_{i=1}^{n} b_{i}.
	\end{aligned}
\end{equation}

To prove Lemma \ref{x4}, we need the following result.

\begin{Lemma}\label{jian}
	Suppose that $\mathbf{w} \in \mathscr{A}$ and $\tau \in(0,1)$. Then  
	\begin{equation}
		\int\limits_{V} \mathrm{e}^{u_{i}^{0}+w_{i}} \mathrm{~d} \mu \leq\left(\frac{\lambda}{4 P_{i}^{2} b_{i} \alpha_{i i}}\right)^{\frac{1-\tau}{\tau}}\left(\int\limits_{V} \mathrm{e}^{\tau u_{i}^{0}+\tau w_{i}} \mathrm{~d} \mu \right)^{\frac{1}{\tau}}, \quad i=1, \ldots, n
	\end{equation}
\end{Lemma}
\begin{proof}
Let $a=\frac{1}{2-\tau}$. By \eqref{69} and Lemma \ref{i}, we conclude that 
	\begin{equation}
		\begin{aligned}[]
				\int\limits_{V} \mathrm{e}^{u_{i}^{0}+w_{i}} \mathrm{~d} \mu &\le \left[ \int\limits_{V} \left( \mathrm{e}^{u_{i}^{0}+w_{i}} \right) ^{\tau}\mathrm{~d} \mu\right] ^{a} \left[ \int\limits_{V} \mathrm{e}^{2(u_{i}^{0}+w_{i})} \mathrm{~d} \mu\right] ^{1-a} \\
				&\le\left[ \int\limits_{V} \left( \mathrm{e}^{u_{i}^{0}+w_{i}} \right) ^{\tau}\mathrm{~d} \mu\right] ^{a}
				\left[ \frac{\lambda}{4 P_{i}^{2} b_{i} \alpha_{i i}}   \left( \int\limits_{V} \mathrm{e}^{u_{i}^{0}+w_{i}} d \mu \right)^{2}      \right] ^{1-a},~i=1,...,n,
		\end{aligned}
	\end{equation}
and hence that

\begin{equation}
	\begin{aligned}
		\int\limits_{V} \mathrm{e}^{u_{i}^{0}+w_{i}} \mathrm{~d} \mu &\le \left( \frac{\lambda}{4 P_{i}^{2} b_{i} \alpha_{i i}} \right) ^{\frac{1-a}{2a-1}} \left( \int\limits_{V} \mathrm{e}^{\tau(u_{i}^{0}+w_{i})} \mathrm{~d} \mu \right) ^{\frac{a}{2a-1}} \\
	&	\le \left( \frac{\lambda}{4 P_{i}^{2} b_{i} \alpha_{i i}} \right) ^{\frac{1-\tau}{\tau}} \left( \int\limits_{V} \mathrm{e}^{\tau(u_{i}^{0}+w_{i})} \mathrm{~d} \mu \right) ^{\frac{1}{\tau}} ,~i=1,...,n.		
	\end{aligned}
\end{equation}
\end{proof}

\begin{Lemma}\label{x4}
	Suppose that $\gamma$ is the smallest eigenvalues of $A$. Then 
	\begin{equation}\label{x}
		J(\mathbf{w}) \geq \frac{\gamma}{4} \sum_{i=1}^{n} \int\limits_{V} \Gamma(w_{i},w_{i}) d \mu-C(\ln \lambda+1),
	\end{equation}
for all $\mathbf{w} \in \mathscr{A}$, where $C>0$ is a constant independent of $\lambda$. 
\end{Lemma}
\begin{proof}
	From \eqref{57} and \eqref{J}, we deduce that 
	\begin{equation}\label{b7}
		J(\mathbf{w}) \geq \frac{\gamma}{2} \sum_{i=1}^{n}\left\|\nabla w_{i}\right\|_{2}^{2}+\sum_{i=1}^{n} b_{i} c_{i}.
	\end{equation}
From \eqref{69} and  \eqref{611}, we deduce that 
\begin{equation}
	\mathrm{e}^{c_{i}} \geq \frac{a_{i}}{2 R_{i} P_{i}  \alpha_{i i} a_{i i}}\ge
		\frac{2 b_{i} P_{i} }{\lambda R_{i} a_{i}}=\frac{2 P_{i} b_{i}}{\lambda R_{i} \int\limits_{V} \mathrm{e}^{u_{i}^{0}+w_{i}} \mathrm{~d} x}, \quad i=1, \ldots, n.
\end{equation}
It follows from that 
\begin{equation}\label{b11}
	c_{i} \geq \ln \frac{2 P_{i} b_{i}}{R_{i}}- \ln \lambda \int\limits_{V} \mathrm{e}^{u_{i}^{0}+w_{i}} \mathrm{~d} x, \quad i=1, \ldots, n.
\end{equation}
By Cauchy inequlity with $\epsilon$($\epsilon$>0) and \eqref{mt},
we deduce that 
\begin{equation}\label{g}
	\begin{aligned}
		\int\limits_{V} e^{w} d\mu &\le \int\limits_{V} e^{\frac{w}{||\nabla w||^{2}_{2}} ||\nabla w||^{2}_{2} } d\mu\\
		&\le \int\limits_{V} e^{\frac{w}{ 4\epsilon||\nabla w||^{2}_{2}}} d\mu e^{ {\epsilon}||\nabla w||_{2}^{2} }\\
		&=:C(\epsilon,G) e^{ \epsilon||\nabla w||^{2}_{2} }.
	\end{aligned}
\end{equation}
It is easy to check that 
\begin{equation}\label{446}
	\begin{aligned}
		||\nabla u^{0}_{i}||^{2}_{2} &= -\int\limits_{V} u_{i}^{0} \Delta u_{i}^{0} d\mu \\
		&= -\int\limits_{V} u^{0}_{i} 4\pi \sum_{s=1}^{N_{i}} \delta_{p_{i s}} d \mu\\
		&=- 4\pi \sum_{s=1}^{N_{i}} u_{i}^{0} (p_{i s }) \\
		&\le 4\pi N_{i} \max_{V} |u_{i}^{0}| \quad i=1, \ldots, n .
	\end{aligned}
\end{equation}
By Lemma \ref{jian}, \eqref{446} and \eqref{g}, we conclude that  
\begin{equation}\label{b1}
	\begin{aligned}
		\ln \int\limits_{V} \mathrm{e}^{u_{i}^{0}+w_{i}} \mathrm{~d} \mu \leq & \frac{1-\tau}{\tau}\left\{\ln \lambda-\ln \left(4 P_{i}^{2} b_{i} \alpha_{i i}\right)\right\}+\frac{1}{\tau} \ln \int\limits_{V} \mathrm{e}^{\tau u_{i}^{0}+s w_{i}} \mathrm{~d} \mu \\
		\leq & 2\epsilon \tau\left\|\nabla w_{i}\right\|_{2}^{2}+ 2\epsilon\tau 4\pi N_{i} \max_{V} |u_{i}^{0}|+ \frac{1-\tau}{\tau}\left\{\ln \lambda-\ln \left(4 P_{i}^{2} b_{i} \alpha_{i i}\right)\right\} \\
		&+\frac{\ln C}{\tau}, \quad i=1, \ldots, n .
	\end{aligned}
\end{equation}
It follows from \eqref{b1}, \eqref{b11} and \eqref{b7} that 
\begin{equation}\label{448}
	\begin{aligned}
		J(\mathbf{w}) \geq &\left(\frac{\gamma}{2}- 2\epsilon \tau \max _{1 \leq i \leq n}\left\{b_{i}\right\} \right) \sum_{i=1}^{n}\left\|\nabla w_{i}\right\|_{2}^{2}-\frac{1}{s} \sum_{i=1}^{n} b_{i}\left\{\ln \lambda-\ln \left(4 P_{i}^{2} b_{i} \alpha_{i i}\right)+\ln C\right\} \\
		&-\sum_{i=1}^{n} b_{i}\left\{\ln \left(2 R_{i} P_{i} b_{i} \alpha_{i i}\right)+ 2\epsilon\tau 4\pi N_{i} \max_{V} |u_{i}^{0}| \right\}
	\end{aligned}.
\end{equation}
Taking $\tau$ sufficiently small in \eqref{448}, we get the desired conclusion \eqref{x}.
\end{proof}

By Lemma \ref{x4}, we can select a minimizing sequence $\{ \mathbf{w}^{k} \}:=\{(w_{1}^{(k)},...,w_{n}^{(k)}) \}$ of the constrained minimization problem 
\begin{equation}
	\eta=\inf\{J(\mathbf{w})| w\in \mathcal{A} \}.
\end{equation}
By Lemma \ref{x4} and Lemma \ref{21}, we see that, there exists $$\mathbf{w}^{0}:=\{(w_{1}^{(0)},...,w_{n}^{(0)}) \} \in H^{0}(V)$$ such that, by passing to a subsequence, denoted still by $\{\mathbf{w}^{k} \}$, 
$${w}_{i}^{(k)} \to {w}_{i}^{(0)} $$ uniformly for $x\in V$ as $k\to +\infty$ for $i=1,...,n$. By the fact that $$\lim\limits_{k\to +\infty} J(\mathbf{w}^{k})=J(\mathbf{w}^{0}),$$ we deduce that $J(\mathbf{w}^{0})=\eta$. Thus $\mathbf{w}^{0}$ is a minimizer of $J$. Next, we prove that the
minimizer belongs to the interior of $\mathscr{A}$.

\begin{Lemma}\label{u8}
	There holds the following inequlities
	\begin{equation}\label{4d}
		\inf _{\mathbf{w} \in \partial \mathscr{A}} J(\mathbf{w}) \geq \frac{|V| \lambda}{2} \min _{1 \leq i \leq n}\left\{\frac{R_{i}}{P_{i}}\right\}-C(1+\ln \lambda+\sqrt{\lambda}),
	\end{equation}
wher $C$ is constant independent of $\lambda$.
\end{Lemma}
\begin{proof}
For any $w\in \partial \mathscr{A}$,	from \eqref{A}, we know that at least one of the following equalities 
\begin{equation}\label{450}
	\frac{a_{i}^{2}}{a_{i i}}=\frac{4 \alpha_{i i} P_{i}^{2} b_{i}}{\lambda}, \quad i=1, \ldots, n
\end{equation}
holds.

Suppose the case $i=1$ happens, from \eqref{713} and \eqref{717}, we deduce that 
\begin{equation}\label{q}
	\begin{aligned}
		a_{1} \mathrm{e}^{c_{1}} & \leq \frac{R_{1}}{P_{1}}\left(Q^{-1}\right)_{11} a_{1}^{2} a_{11}^{-1}+\sum_{j=2}^{n} \frac{R_{j}}{P_{j}}\left(Q^{-1}\right)_{1 j} a_{1} a_{j} a_{11}^{-\frac{1}{2}} a_{j j}^{-\frac{1}{2}} \\
		& \leq \frac{R_{1}}{P_{1}}\left(Q^{-1}\right)_{11} a_{1}^{2} a_{11}^{-1}+|V|^{\frac{1}{2}} \sum_{j=2}^{n} \frac{R_{j}}{P_{j}}\left(Q^{-1}\right)_{1 j} a_{1} a_{11}^{-\frac{1}{2}}.
	\end{aligned}
\end{equation}
By \eqref{q} and \eqref{450}, we deduce that 
\begin{equation}\label{115}
	a_{1} \mathrm{e}^{c_{1}} \le  \left(Q^{-1}\right)_{11} \frac{4 P_{1} R_{1} b_{1} \alpha_{11}}{\lambda}+\frac{2 P_{1} \sum_{j=2}^{n} \frac{R_{j}}{P_{j}}\left(Q^{-1}\right)_{1 j}}{\sqrt{\lambda}} \sqrt{b_{1}|\textsc{V}| \alpha_{11}}.
\end{equation}
If other cases occur, we can establish the similar estimate.
From \eqref{76} and \eqref{115}, we deduce that 
\begin{equation}\label{b15}
	\begin{aligned}
		&\frac{\lambda}{2} \sum_{i=1}^{n} \frac{R_{i}}{P_{i}} \int\limits_{V}\left(1-\mathrm{e}^{c_{i}} \mathrm{e}^{u_{i}^{0}+w_{i}}\right) \mathrm{d} \mu \\
		&\geq \frac{|V| \lambda R_{1}}{2 P_{1}}-2\left(Q^{-1}\right)_{11} P_{1} R_{1} b_{1} \alpha_{11}-P_{1} \sum_{j=2}^{n} \frac{R_{j}}{P_{j}}\left(Q^{-1}\right)_{1 j} \sqrt{b_{1} \lambda|V| \alpha_{11}}.
	\end{aligned}
\end{equation}
From \eqref{b1}, \eqref{b11} and \eqref{b15}, we obtain the desired conclusion \eqref{4d}.
\end{proof}

\begin{Lemma}\label{f1}
	There exists $\mathbf{w}_{r_{\varepsilon}}\in int \mathscr{A}$ such that
	\begin{equation}
		J\left(\mathbf{w}_{r_{\varepsilon}}\right)-\inf _{\mathbf{w} \in \partial \mathscr{A}} J(\mathbf{w})<-1.
	\end{equation}
\end{Lemma}
\begin{proof}
 	By Proposition \ref{f2}, we see that, for all sufficiently large $r>0$, the problem 
 	\begin{equation}
 		\Delta v=r \mathrm{e}^{u_{i}^{0}+v}\left(\mathrm{e}^{u_{i}^{0}+v}-1\right)+\frac{4 \pi N_{i}}{|V|}, \quad i=1, \ldots, n,
 	\end{equation}
 has solutions $v_{i,r}(i=1,...,n)$ 	so that $v_{i,r} \to -u_{i}^{0}$ as $r\to +\infty$ uniformly for $x\in V$. Let $c_{i,r}:=\frac{1}{|V|}\int\limits_{V} v_{i,r} d\mu$. It follows that $w_{i,r}:=v_{i,r}-c_{i,r} \to -u_{i}^{0}$ as $r\to +\infty$, $i=1,...,n.$ Thus we have 
 	\begin{equation}\label{456}
 		\lim _{\mu \rightarrow \infty} a_{i}(w_{i,r})=|V|, \quad \lim _{\mu \rightarrow \infty}a_{ij}( w_{i,r},w_{j,r} )=|V|, \quad i, j=1, \ldots, n.
 	\end{equation}
 	By \eqref{62}, we obtain 
 	\begin{equation}\label{457}
 		\lim _{r \rightarrow \infty} \tilde{Q}\left(\mathbf{w}_{r}\right)=|V| Q .
 	\end{equation}
 By \eqref{457} and \eqref{456}, we can find $\sigma>0$ such that for any $\epsilon\in (0,1)$, there exists $r_{\epsilon}>0$ so that 
 \begin{equation}
 	\mathbf{w}_{r_{\varepsilon}}=\left(w_{1,r_{\varepsilon}}, \ldots, w_{n,r_{\varepsilon}}\right)^{T} \in \text { int } \mathscr{A}
 \end{equation}	
 	for all $\lambda>\sigma$, and \begin{equation}\label{459}
 		\begin{aligned}
 			&a_{i j}\left(w_{i,r_{\varepsilon}}, w_{j,r_{\varepsilon}}\right)<(1+\varepsilon)|V|<2|V|, \quad i, j=1, \ldots, n \\
 			&\frac{(1-\varepsilon)}{|V|} Q^{-1}<\tilde{Q}^{-1}\left(\mathbf{w}_{r_{\varepsilon}}\right)<\frac{(1+\varepsilon)}{|V|} Q^{-1}<\frac{2}{|V|} Q^{-1}.
 		\end{aligned}
 	\end{equation}                Since $\mathbf{w}_{r_{\varepsilon}}\in \text { int } \mathscr{A}$, by \eqref{67} and the fact that
 $$\sqrt{1-2x}\ge 1-2x\text{~for~}x\in[0,\frac{1}{2}] ,$$   we deduce that \begin{equation}
 	\begin{aligned}
 		&\mathrm{e}^{c_{i}\left(\mathbf{w}_{r_{\varepsilon}}\right)}=\frac{\frac{R_{i} a_{i}}{P_{i}}+\sum\limits_{j \neq i} \mathrm{e}^{c_{j}\left(\mathbf{w}_{r_{\varepsilon}}\right)} R_{i} R_{j} \alpha_{i j} a_{i j}}{2 R_{i}^{2} \alpha_{i i} a_{i i}}\\
 		&\times\left(1+\sqrt{1-\frac{4 b_{i} R_{i}^{2} \alpha_{i i} a_{i i}}{\lambda\left(\frac{R_{i} a_{i}}{P_{i}}+\sum\limits_{j \neq i} \mathrm{e}^{\left.c_{j} R_{i} R_{j} \alpha_{i j} a_{i j}\right)^{2}}\right.}}\right)\\
 		&\geq \frac{\frac{R_{i} a_{i}}{P_{i}}+\sum_{j \neq i} \mathrm{e}^{c_{j}\left(\mathbf{w}_{r_{\varepsilon}}\right)} R_{i} R_{j} \alpha_{i j} a_{i j}}{R_{i}^{2} \alpha_{i i} a_{i i}}-\frac{2 b_{i}}{
 			\lambda\left(\frac{R_{i} a_{i}}{P_{i}}+\sum\limits_{j \neq i} \mathrm{e}^{c_{j} R_{i} R_{j} \alpha_{i j} a_{i j}}
 			\right)
 		}.
 	\end{aligned}
 \end{equation}  
 By \eqref{k1} and Jensen inequlity, we conclude that 	
 \begin{equation}\label{w}
 	\mathrm{e}^{c_{i}\left(\mathbf{w}_{r_{\varepsilon}}\right)}
 	\geq \frac{\frac{R_{i}|V|}{P_{i}}+\sum\limits_{j \neq i} \mathrm{e}^{c_{j}\left(\mathbf{w}_{r_{\varepsilon}}\right)} R_{i} R_{j} \alpha_{i j} a_{i j}}{R_{i}^{2} \alpha_{i i} a_{i i}}-\frac{2 P_{i} b_{i}}{\lambda|V| R_{i}}, \quad i=1, \ldots, n.
 \end{equation}	
 From now on, we understand
 \begin{equation}
 	a_{i}=a_{i}\left(w_{i,r_{\varepsilon}} \right), \quad a_{i j}=a_{i j}\left(w_{i,r_{\varepsilon}}, w_{j,r_{\varepsilon}}\right), \quad i, j=1, \ldots, n.
 \end{equation}	

Then by \eqref{w} and \eqref{459}, we deduce that 
\begin{equation}\label{eb}
	\begin{aligned}
		R_{i}^{2} \alpha_{i i} a_{i i} \mathrm{e}^{c_{i}\left(\mathrm{w}_{r_{\varepsilon}}\right)}-\sum_{j \neq i} \mathrm{e}^{c_{j}\left(\mathbf{w}_{r_{\varepsilon}}\right)} R_{i} R_{j} \alpha_{i j} a_{i j} & \geq \frac{R_{i}|V|}{P_{i}}-\frac{2 P_{i} R_{i} b_{i}}{\lambda|V|} \alpha_{i i} a_{i i} \\
		& \geq \frac{R_{i}|V|}{P_{i}}-\frac{4 \alpha_{i i} P_{i} R_{i} b_{i}}{\lambda}, \quad i=1, \ldots, n.
	\end{aligned}
\end{equation}
By \eqref{459} and \eqref{eb}, we conclude that 
\begin{equation}\label{l}
	\begin{aligned}
		&\left(\mathrm{e}^{c_{1}\left(\mathbf{w}_{r_{\varepsilon}}\right)}, \ldots, \mathrm{e}^{c_{n}\left(\mathbf{w}_{r_{\varepsilon}}\right)}\right)^{T} \\
		&\geq|\Omega| \tilde{Q}^{-1}\left(\mathbf{w}_{r_{\varepsilon}}\right) P^{-1} R \mathbf{1}-\frac{4 \tilde{Q}^{-1}\left(\mathbf{w}_{r_{\varepsilon}}\right) P R}{\lambda} \operatorname{diag}\left\{\alpha_{11}, \ldots, \alpha_{n n}\right\} \mathbf{b} \\
		&\geq(1-\varepsilon) Q^{-1} P^{-1} R \mathbf{1}-\frac{8 Q^{-1} P R}{\lambda|\Omega|} \operatorname{diag}\left\{\alpha_{11}, \ldots, \alpha_{n n}\right\} \mathbf{b} \\
		&=(1-\varepsilon) \mathbf{1}-\frac{8 Q^{-1} P R}{\lambda|\Omega|} \operatorname{diag}\left\{\alpha_{11}, \ldots, \alpha_{n n}\right\} \mathbf{b} .
	\end{aligned}
\end{equation}
It follows that 
\begin{equation}
	\int\limits_{V}\left(1-\mathrm{e}^{c_{i}\left(\mathbf{w}_{r_{\varepsilon}}\right)} \mathrm{e}^{u_{i}^{0}+w_{i}^{\mu_{\varepsilon}}}\right) \mathrm{d} \mu \leq|V| \varepsilon+\frac{8}{\lambda} \sum_{i=1}^{n}\left(Q^{-1}\right)_{i j} P_{j} R_{j} b_{j} \alpha_{j j}, \quad i=1, \ldots, n.
\end{equation}
By \eqref{77}, there exists a constant $C_{\epsilon}$ such that
	\begin{equation}
		J\left(\mathbf{w}_{r_{\varepsilon}}\right) \leq \frac{|V| \lambda \varepsilon}{2} \sum_{j=1}^{n} \frac{R_{i}}{P_{i}}+C_{\varepsilon}.
	\end{equation}
By Lemma \ref{u8}, this implies that there exists $C$ independent of $\lambda$ so that
\begin{equation}\label{f}
	J\left(\mathbf{w}_{r_{\varepsilon}}\right)-\inf _{\mathbf{w} \in \partial \mathscr{A}} J(\mathbf{w}) \leq \frac{|V| \lambda}{2}\left(\sum_{j=1}^{n} \frac{R_{i}}{P_{i}} \varepsilon-\min _{1 \leq i \leq n}\left\{\frac{R_{i}}{P_{i}}\right\}\right)+C(\sqrt{\lambda}+\ln \lambda+1).
\end{equation}
We can get \eqref{f1} by taking $\epsilon$ suitably small and $\lambda$ sufficiently large in \eqref{f}.

We now prove the lemma.
\end{proof}
From Lemmas \ref{x4} and \ref{f1}, there exists $\lambda_{2}:= \max\{ \sigma,\lambda_{0} \}$ such that for all $\lambda>\lambda_{2}$, we can find  $\mathbf{w}_{0} \in int \mathscr{A}$ such that $\mathbf{w}_{0}$ is a minimizer of $J$.   It is easy to check that $\mathbf{v}_{0}:= \mathbf{w}_{0}+ \mathbf{c}({\mathbf{w}_{0}})$ is a critical point of $I$, which implies that $\mathbf{v}_{0}$ is a solution of equations \eqref{510}. Thus, we could establish the conclusion (ii) of Theorem \ref{31}.

\end{document}